\documentclass[12pt, a4paper]{article}
\usepackage[utf8]{inputenc}
\usepackage{bbm}

\usepackage[margin = 1in]{geometry}
\usepackage{amsfonts,amsmath,amssymb,amsthm}
\usepackage{graphicx,subcaption}
\usepackage{booktabs}
\usepackage{lipsum}
\usepackage{hyperref}
\usepackage{multicol}
\hypersetup{
   colorlinks = true,
   citecolor = blue,
   urlcolor = purple,
   linkcolor = red
}

\usepackage{mathrsfs}
\usepackage{multirow}
\usepackage[makeroom]{cancel}
\usepackage{tikz}
\usepackage{framed,color}
\definecolor{shadecolor}{rgb}{1,0.8,0.3}
\usepackage{fancybox}
\usepackage{caption}
\usepackage{algorithm,algorithmic} 
\usepackage{multicol}
\usepackage{aliascnt}
\usepackage{setspace}
\usepackage{longtable}
\allowdisplaybreaks

\usetikzlibrary{decorations.markings,arrows}

\allowdisplaybreaks

\title{\textbf{Some Results on Generalized Familywise Error Rate Controlling Procedures under Dependence}}

\DeclareMathOperator*{\argmax}{arg\,max}

\date{}
\vspace{5mm}
\usepackage{authblk}
\author[1]{Monitirtha Dey\footnote{\href{mdey@uni-bremen.de}{mdey@uni-bremen.de}}}
\author[2]{Subir Kumar Bhandari\footnote{\href{subirkumar.bhandari@gmail.com}{subirkumar.bhandari@gmail.com}}}
\affil[1]{\small Institute for Statistics, University of Bremen, Bremen, Germany}
\affil[2]{\footnotesize Interdisciplinary Statistical Research Unit, Indian Statistical Institute, Kolkata, India}
 
\usepackage{natbib}
\setcitestyle{authoryear, open={(},close={)}}

\begin{document}

\maketitle
\theoremstyle{plain}
\newtheorem{axiom}{Axiom}
\newtheorem{remark}{Remark}
\newtheorem{corollary}{Corollary}[section]
\newtheorem{claim}[axiom]{Claim}
\newtheorem{theorem}{Theorem}[section]
\newtheorem{lemma}{Lemma}[section]

\newaliascnt{lemmaa}{theorem}
\newtheorem{lemmaa}[lemmaa]{Theorem}
\aliascntresetthe{lemmaa}
\providecommand*{\lemmaaautorefname}{Theorem}



\theoremstyle{plain}
\newtheorem{exa}{Example}
\newtheorem{rem}{Remark}
\newtheorem{proposition}{Proposition}

\theoremstyle{definition}
\newtheorem{definition}{Definition}
\newtheorem{example}{Example}







 \begin{abstract}
 The topic of multiple hypotheses testing now has a potpourri of novel theories and ubiquitous applications in diverse scientific fields. However, the universal utility of this field often hinders the possibility of having a generalized theory that accommodates every scenario. This tradeoff is better reflected through the lens of dependence, a central piece behind the theoretical and applied developments of multiple testing. Although omnipresent in many scientific avenues, the nature and extent of dependence vary substantially with the context and complexity of the particular scenario. Positive dependence is the norm in testing many treatments versus a single control or in spatial statistics. On the contrary, negative dependence arises naturally in tests based on split samples and in cyclical, ordered comparisons. In GWAS, the SNP markers are generally considered to be weakly dependent. Generalized familywise error rate ($k$-FWER) control has been one of the prominent frequentist approaches in simultaneous inference. However, the performances of $k$-FWER controlling procedures are yet unexplored under different dependencies. This paper revisits the classical testing problem of normal means in different correlated frameworks. We establish upper bounds on the generalized familywise error rates under each dependence, consequently giving rise to improved testing procedures. Towards this, we present improved probability inequalities, which are of independent theoretical interest. 

 \end{abstract}
\newpage
\section{Introduction\label{sec:1}}

The field of simultaneous statistical inference has witnessed a beautiful blend of theory and applications in its development. Now, we have a potpourri of novel theories and a stream of wide-reaching applications in diverse scientific fields. However, the universal utility of this topic often hinders the possibility of having a generalized theory that fits every situation. This tradeoff is better reflected through the lens of dependence, a central piece behind the theoretical and applied developments of multiple testing. 

While dependence is a natural phenomenon in a plethora of scientific avenues, the nature and extent of dependence varies with the context and complexity of the particular scenario: 
\begin{enumerate}
    \item Positively correlated observations (e.g., the equicorrelated setup) arise in several applications, e.g., when comparing a control against several treatments. Consequently, numerous recent works in multiple testing consider the equicorrelated setup and the positively correlated setup\citep{Delattre, deycstm, deybhandari, deybhandaristpa, Proschan, royspl}. 
    \item Negative dependence also naturally appears in many testing and multiple comparison scenarios \citep{ChiRamdasWang, Joag-Dev}, e.g., in tests based on split samples, and in cyclical, ordered comparisons. 
    \item The correlation between two single nucleotide polymorphisms (SNP) is thought \citep{Proschan} to decrease with genomic distance. Many authors have argued that for the large numbers of markers typically used for a GWA study, the test statistics are weakly correlated because of this largely local presence of correlation between SNPs. \cite{StoreyTib2003} define weak dependence as “any form of dependence whose effect becomes negligible as the number of features increases to infinity” and remark that weak dependence generally holds in genome-wide scans.
\end{enumerate}
Generalized familywise error rate has been one of the most prominent approaches in frequentist simultaneous inference. However, the
performances of $k$-FWER controlling procedures are yet unexplored under different dependencies. This paper revisits the classical testing problem of normal means in different correlated frameworks. We establish upper bounds on the generalized familywise error rates under each dependence, consequently giving rise to improved testing procedures. Towards this, we also present improved probability inequalities, which are of independent theoretical interest. 

This paper is organized as follows. We introduce the testing framework with relevant notations and summarize some results on the limiting behavior of the Bonferroni method under the equicorrelated normal setup in the next section. Section \ref{sec:3} presents improved bounds on $k$-FWER under independence and negative dependence of the test statistics.  We obtain new and improved probability inequalities in \ref{sec:4}. Section \ref{sec:5} employs these results to our multiple testing framework. We consider the nearly independent set-up in Section \ref{sec:6}. Results of an empirical study and real data analysis are presented in Section \ref{sec:7} before we conclude with a brief discussion in \ref{sec:8}.

\section{The Framework\label{sec:2}}

We address the multiple testing problem through a \textit{Gaussian sequence model}: 
$$X_{i} {\sim} N(\mu_{i},1),  \quad i \in \{1, \ldots,n\}$$
where $\operatorname{Corr}\left(X_i, X_j\right)=\rho_{ij}$ for each $i \neq j$. We wish to test: 
$$H_{0i}:\mu_{i}=0 \quad vs \quad H_{1i}:\mu_{i}>0, \quad 1 \leq i \leq n.$$ The global null  $ H_{0}=\bigcap_{i=1}^{n} H_{0 i}$ hypothesizes that each mean is zero. In the following, $\Sigma_n$ denotes the correlation matrix of $X_1, \ldots, X_n$ with $(i,j)$’th entry $\rho_{ij}$.

The usual Bonferroni procedure uses the cutoff $\Phi ^{-1}(1-\alpha/n)$ to control FWER at level $\alpha$. ~\cite{Lehmann2005} remark that controlling $k$-FWER allows one to decrease this cutoff to $\Phi ^{-1}(1-k\alpha/n)$, and thus significantly increase the ability to identify false hypotheses. Thus, for their Bonferroni-type procedure, under the global null,
\begin{align*}k\text{-FWER}(n, \alpha, \Sigma_n) &=\mathbb{P}_{\Sigma_n}\left(X_{i}>\Phi ^{-1}(1-k\alpha/n) \hspace{2mm} \text{for at least $k$}
\hspace{2mm}\text{$i$'s} \mid H_{0}\right).\end{align*}
Evidently, when $k=1$, Lehmann-Romano procedure simplifies to the Bonferroni method and $k$-FWER reduces to the usual FWER.

In \cite{deybhandaristpa}, we proposed a multiple testing procedure that controls generalized FWER under non-negative dependence. The generalized FWER for this procedure is given by
$$k\text{-FWER}_{modified}\left(n, \alpha, \Sigma_{n}\right)=\mathbb{P}_{\Sigma_n}\left(X_{i}>\Phi ^{-1}(1-k\alpha^{\star}/n) \hspace{2mm} \text{for at least $k$}
\hspace{2mm}\text{$i$'s} \mid H_{0}\right)$$
where $$\alpha^{\star} := \argmax_{\beta\in (0,1)} \Bigl\{\min\{f_{n, k, \Sigma_{n}}(\beta), g_{n, k, \Sigma_{n}}(\beta)\} \leq \alpha\Bigl\}.$$ Here $f$ and $g$ are functions defined as follows:
\begin{align*}&f_{n, k, \Sigma_{n}}(\alpha) &= \frac{(n-1)k}{n(k-1)} \cdot \alpha^2 + \frac{1}{\pi k(k-1)} \sum_{1 \leq i <j < n} \int_{0}^{\rho_{ij}}\dfrac{1}{\sqrt{1-z^2}}e^{\frac{-\{\Phi ^{-1}(1-\frac{k\alpha}{n})\}^2}{1+z}}dz,\end{align*}
\begin{align*}g_{n, k, \Sigma_{n}}(\alpha) 
& = \alpha \cdot \frac{n+k-1}{n} - \dfrac{n-1}{n}\cdot \dfrac{{k\alpha}^2}{n}- \dfrac{1}{2\pi k} \sum_{j=1, j\neq i^{\star}}^{n} \int_{0}^{\rho_{i^{\star}j}}\dfrac{1}{\sqrt{1-z^2}}e^{\frac{-\{\Phi ^{-1}(1-\frac{k \alpha}{n})\}^2}{1+z}}dz.\end{align*}
where $i^{\star} =\argmax_{i} \sum_{j=1, j\neq i}^{n} \rho_{ij}$.

\noindent For each $k>1$, \cite{deybhandaristpa} show that 
$$k\text{-FWER}_{modified}(n, \alpha, \Sigma_n) \leq \min\{f_{n, k, \Sigma_{n}}(\alpha), g_{n, k, \Sigma_{n}}(\alpha)\}.$$

\noindent Therefore, $k\text{-FWER}_{modified}(n, \alpha, \Sigma_n)$ is not more than $\alpha$ when
$$\alpha^{\star} := \argmax_{\beta\in (0,1)} \Bigl\{\min\{f_{n, k, \Sigma_{n}}(\beta), g_{n, k, \Sigma_{n}}(\beta)\} \leq \alpha\Bigl\}.$$

\noindent Since $\alpha^{\star} \geqslant \alpha$, their procedure improves the Lehmann-Romano method. In \cite{deybhandaristpa}, the authors establish the following result:

\begin{theorem}\label{STPAlemma4.1}
Let $k >1$. Suppose  $X_1, \ldots, X_n$ are independent. Moreover, let $\alpha \leq \frac{n(k-1)}{(n-1)k}$. Then, 
$\alpha^{\star} = \sqrt{\frac{n(k-1)\alpha}{(n-1)k}}.$
\end{theorem}

This result implies that, under the independence of the test statistics, $\alpha^{\star}$ can be chosen close to $\sqrt{\alpha}$. This actually greatly increases the ability to reject false null hypotheses. Naturally the following question arises:

\vspace{3mm}

\noindent \textit{Do there exist nontrivial constants $D_{n, k}$ such that one has}
$$\mathbb{P}_{\Sigma_n}\left(X_{i}>\Phi ^{-1}(1-k\alpha^{\star}/n) \hspace{2mm} \text{for at least $k$}
\hspace{2mm}\text{$i$'s} \mid H_{0}\right) \leq \alpha$$
\textit{with some $\alpha^{\star} \geqslant D_{n, k} \cdot\alpha^{1 / k} ?$}

\vspace{3mm}
If this is answered in affirmative, then we would have sharper upper bounds on generalized FWERs and consequently, improved multiple testing procedures. Otherwise, can one devise improved multiple testing procedures, using the theory of probability inequalities?



\section{Improved Bounds under Independence and Negative Dependence\label{sec:3}}
 For a random vector $\mathbf{T}=$ $\left(T_1, \ldots, T_n\right)$, let $F_k$ be the distribution function of $T_k$ for $k \in [n]:=\{1, \ldots,n\}$. $\mathbf{T}$ is called (lower) weakly negatively dependent if
$$
\mathbb{P}\left(\bigcap_{k \in A}\left\{T_k \leq F_k^{-1}(p)\right\}\right) \leq \prod_{k \in A} \mathbb{P}\left(T_k \leq F_k^{-1}(p)\right) \quad \text { for all } A \subseteq [n] \text { and } p \in(0,1).
$$
Note that when there is exact equality for each $A$ and each $p$, we have independence. We consider the means testing problem as described in section \ref{sec:2} but with a general underlying distribution $F$.
\begin{theorem}\label{bound1}
    Suppose  $\mathbf{X}=(X_1, \ldots, X_n)$ is negatively dependent where $X_{i}\sim F$ with the density of $F$ being symmetric about zero. Then, $$\mathbb{P}_{\Sigma_n}\left(X_{i}>F ^{-1}(1-k\alpha^{\star}/n) \hspace{2mm} \text{for at least $k$}
\hspace{2mm}\text{$i$'s} \mid H_{0}\right) \leq \alpha$$
where $$
\alpha^{\star}=\left[\frac{n}{k \cdot\binom{n}{k}^{1 / k}}\right] \alpha^{1 / k}.
$$
\end{theorem}
\begin{proof}
    Let $A_{i}$ denote the event $\left\{X_{i}>F^{-1}\left(1-\frac{k \alpha^{\star}}{n}\right)\right\} \equiv \left\{-X_{i}<-F^{-1}\left(1-\frac{k \alpha^{\star}}{n}\right)\right\}\equiv \left\{-X_{i}<F^{-1}\left(\frac{k \alpha^{\star}}{n}\right)\right\}$.
Then $\mathbb{P}_{H_{0}}\left(A_{i}\right)=k\alpha^{\star}/n$.
Now,
\begin{align*}
k\text{-FWER}_{modified}\left(n, \alpha, I_{n}\right) & =\mathbb{P}\left(\text{at least $k$ many $A_{i}$'s occur}\right) \\
& =\mathbb{P}\left(\bigcup_{i_{1}, \ldots, i_{k}}\left\{A_{i}, \cap \ldots \cap A_{i_{k}}\right\}\right) \\
& \leqslant \binom{n}{k} \sum_{i_{1}, \ldots, i_{k}} \mathbb{P}\left(A_{i_{1}} \cap \ldots . \cap A_{i_{k}}\right) \quad \text{(using Boole's inequality)}\\
& \leq \binom{n}{k} \cdot\left(\frac{k \alpha^{\star}}{n}\right)^{k} \quad (\text{from negative dependence}).
\end{align*}
Now, $\binom{n}{k} \cdot\left(\frac{k \alpha^{\star}}{n}\right)^{k} \leqslant \alpha$ gives
$$
\frac{k \alpha^{\star}}{n} \leqslant\left\{\frac{\alpha}{\binom{n}{k}}\right\}^{1/k}.
$$
The rest follows. 
\end{proof}

When $(X_1, \ldots, X_n)$ are independent then upper bounds on generalized FWER may also be obtained through the Chernoff bound:
\begin{theorem}
    Let $Y_1, \ldots, Y_n$ be independent Bernoulli($p$) random variables. Suppose $Y=\sum_{i=1}^{n}Y_{i}$. Then, 
    $$\mathbb{P}(Y \geq a) \leq \inf_{t >0} e^{-ta} {(1-p+pe^t)}^n.$$
\end{theorem}
\noindent Putting $a=(1+\delta)np$ (for $\delta>0$) and $t=\log \left[(1+\delta)/p\right]$, one obtains the following result:
\begin{theorem}\label{chernoff2}
    Let $Y_1, \ldots, Y_n$ be independent Bernoulli($p$) random variables. Suppose $Y=\sum_{i=1}^{n}Y_{i}$. Then, for any $\delta>0$,
    $$\mathbb{P}\big[Y \geq (1+\delta)np\big] \leq \Bigg[ \frac{e^\delta}{(1+\delta)^{1+\delta}}\Bigg]^{np}.$$
\end{theorem}
\noindent In our multiple testing context, let us consider the Bernoulli random variables 
$$Y_{i} = \mathbbm{1}\left\{ X_{i} > F^{-1}(1-k\alpha^{\star}/n)\right\}, \quad 1\leq i \leq n.$$
Then, under the global null, $\mathbb{P}(Y_{i}=1)=k\alpha^{\star}/n$ for each $i$. We observe that\\ $k\text{-FWER}_{modified}\left(n, \alpha, I_{n}\right)$ is same as $\mathbb{P}(\sum_{i=1}^{n}Y_{i} \geq k)$.

\vspace{3mm}

Consider $\delta=1/\alpha^{\star}-1$. Then, \autoref{chernoff2} gives
$$\mathbb{P}\left(\sum_{i=1}^{n}Y_{i} \geq k\right) \leq \Bigg[ \frac{{e}^{\frac{1}{\alpha^{\star}}-1}}{{(\frac{1}{\alpha^{\star}})}^{\frac{1}{\alpha^{\star}}}}\Bigg]^{k \alpha^{\star}}= \left[e^{1-\alpha^{\star}}\cdot \alpha^{\star}\right]^{k}.$$

\noindent We wish to use $\alpha^{\star}$ for which $\left[e^{1-\alpha^{\star}}\cdot \alpha^{\star}\right]^{k}\leq \alpha$. It is sufficient to have ${(e \alpha^{\star})}^{k} \leq \alpha$. In other words, we may choose $\alpha^{\star}= \frac{1}{e}\cdot {\alpha}^{1/k}$. Hence, we obtain the following result:
\begin{theorem}\label{bound2}Suppose  $X_1, \ldots, X_n$ are independent. Then, $$\mathbb{P}_{\Sigma_n}\left(X_{i}>F ^{-1}(1-k\alpha^{\star}/n) \hspace{2mm} \text{for at least $k$}
\hspace{2mm}\text{$i$'s} \mid H_{0}\right) \leq \alpha$$
where $$\alpha^{\star}= \frac{1}{e}\cdot {\alpha}^{1/k}.$$
\end{theorem}

We now compare the bounds obtained through Boole's inequality and Chernoff's bound. Towards this, we note
$$
\binom{n}{k}=\frac{n!}{(n-k)!k!}=\frac{n \cdot(n-1) \ldots(n-(k-1))}{k!} \leq \frac{n^k}{k!}.
$$
On the other hand, $e^k=\sum_{i=0}^{\infty} \frac{k^i}{i!} >\frac{k^k}{k!}$. This implies $\frac{1}{k!}<\left(\frac{e}{k}\right)^k$. Substituting this in the above inequality, we get 
$$\binom{n}{k} \leq \left(\frac{en}{k}\right)^k.$$
This results in 
$$\frac{1}{e} \leq \frac{n}{k \cdot\binom{n}{k}^{1 / k}}.$$
Hence \autoref{bound1} is stronger than \autoref{bound2}. 

\begin{remark}
    We have previously noted that 
$k\text{-FWER}_{modified}\left(n, \alpha, I_{n}\right)$ is same as $\mathbb{P}(\sum_{i=1}^{n}Y_{i} \geq k)$ where $Y_{i} = \mathbbm{1}\left\{ X_{i} > \Phi^{-1}(1-k\alpha^{\star}/n\right\}$, for $1\leq i \leq n.$ Hoeffding's inequalty gives 
$$\mathbb{P} \Bigg[\sum_{i=1}^{n} Y_{i} \geq \mathbb{E}\left(\sum_{i=1}^{n} Y_{i}\right) +\delta\Bigg] \leq e^{-2 \delta^2/n}.$$
Putting $\delta=k(1-\alpha^{\star})$ results in 
$$k\text{-FWER}_{modified}\left(n, \alpha, I_{n}\right) \leq e^{-2 k^2 (1-\alpha^{\star})^2/n}.$$
This implies, whenever $k^2/n \to \infty$ as $n \to \infty$, $k\text{-FWER}_{modified}\left(n, \alpha, I_{n}\right)$ approaches zero. In particular, for any sequence $k_n = n^{\beta}$ with $\beta >1/2$, $k\text{-FWER}_{modified}\left(n, \alpha, I_{n}\right)$ approaches zero. Also note that, this is true under general distributions, since its proof actually nowhere uses normality of the test statistics.
\end{remark}

\section{A New Probability Inequality\label{sec:4}}
We have previously observed that $k$-FWER is $\mathbb{P}(\text{at least}\hspace{2mm} k \hspace{1.5mm}\text{out of}\hspace{1.5mm} n \hspace{1.5mm} A_{i}\text{'s occur})$ for suitably defined events $A_i$, $1 \leq i \leq n$. Naturally one wonders whether the theory of probability inequalities helps to find improved upper bounds on $\mathbb{P}(\text{at least}\hspace{2mm} k \hspace{1.5mm}\text{out of}\hspace{1.5mm} n \hspace{1.5mm} A_{i}\text{'s occur})$. Accurate computation of this probability requires knowing the complete dependence between the events $(A_1,\ldots, A_n)$, which we typically do not know unless they are independent. As also mentioned in \cite{deycstm} and \cite{deybhandaristpa}, the available information is often the marginal probabilities and joint probabilities up to level $k_{0}(k_{0}<<n)$. Towards finding such an easily computable upper bound on the probability that at least $k$ out of $n$ events occur, \cite{deybhandaristpa} establish in the following:
\begin{theorem}
    Let $A_1$, $A_2$, \ldots, $A_n$ be $n$ events. Then, for each $k \geq 2$,
$$\mathbb{P}(\text{at least}\hspace{2mm} k \hspace{1.5mm}\text{out of}\hspace{1.5mm} n \hspace{1.5mm} A_{i}\text{'s occur}) \leq \min \Biggl\{\frac{S_1-S^{\prime}_2}{k} + \frac{k-1}{k}\cdot \max_{1 \leq i \leq n}\mathbb{P}(A_i), \frac{2S_2}{k(k-1)}\Biggl\}$$where
$S_{1} = \sum_{i=1}^{n} \mathbb{P}(A_i)$, $S_{2} = \sum_{1 \leq i <j < n} \mathbb{P}(A_i \cap A_j)$ and $$S^{\prime}_{2} = \max_{1 \leq i \leq n} \sum_{j=1, j\neq i}^{n} \mathbb{P}(A_i \cap A_j).$$
\end{theorem}
\noindent We wish to obtain sharper probability inequalities. Towards this, we define some quantities. For $1\leq m \leq n$, suppose $$S_{m} = \sum_{1 \leq i_1 < \cdots< i_m \leq n} \mathbb{P}(A_{i_{1}} \cap \cdots \cap  A_{i_m})$$ denotes the sum of probabilities of $m$-wise intersections. Also, let 
 $$S^{\prime}_{m} = \max_{i_1 < \ldots < i_{m-1}} \sum_{j=1, j \notin \{ i_1, \ldots, i_{m-1}\}}^{n} \mathbb{P}(A_j \cap A_{i_1} \ldots \cap A_{i_{m-1}}).$$

 \begin{lemma}
 \label{STPALemma3.1}
Let $A_1$, $A_2$, \ldots, $A_n$ be $n$ events.  Given any $k\geq 2$, we have
$$\mathbb{P}(\text{at least}\hspace{2mm} k \hspace{1.5mm}\text{out of}\hspace{1.5mm} n \hspace{1.5mm} A_{i}\text{'s occur}) \leq \frac{S_1-S^{\prime}_m}{k} + \frac{k-m+1}{k}\cdot \max_{i_1 < \ldots < i_{m-1}}\mathbb{P}(A_{i_1}\cap \ldots \cap A_{i_{m-1}})$$
for each $2\leq m\leq k$.
 \end{lemma}
 \begin{proof}
     Let $I_{i}(w)$ be the indicator random variable of the event $A_i$ for $1 \leq i \leq n$. Then the random variable $\max I_{i_{1}}(w)\cdots I_{i_{k}}(w)$ is the indicator of the event that at least $k$ among $n$ $A_i$'s occur. Here the maximum is taken over all tuples $(i_1, \ldots, i_k)$ with $i_{1}, \ldots, i_k \in \{1,\ldots,n\}$, $i_1<\ldots< i_k$. Now, for any $i = 1, \ldots, n$,
\begin{align*}
    \max I_{i_{1}}(w)\cdots I_{i_{k}}(w) & \leq \frac{1}{k}[1-I_{i_1}(w)\cdots I_{i_{m-1}}(w)]\sum_{j=1}^{n}I_{j}(w) + I_{i_1}(w)\cdots I_{i_{m-1}}(w).
\end{align*}
Taking expectations in above, we obtain
\begin{align*}
    & \mathbb{P}(\text{at least}\hspace{2mm} k \hspace{1.5mm}\text{out of}\hspace{1.5mm} n \hspace{1.5mm} A_{i}\text{'s occur}) \\
   \leq  & \dfrac{1}{k}\cdot \sum_{j=1}^n \mathbb{P}(A_j)-\frac{1}{k}\cdot \sum_{j=1, j \notin \{ i_1, \ldots, i_{m-1}\}}^{n} \mathbb{P}(A_j \cap A_{i_1} \ldots A_{i_{m-1}}) + \mathbb{P}( A_{i_1} \cap \ldots \cap A_{i_{m-1}}) \cdot \dfrac{k-m+1}{k}.
\end{align*}
The rest follows by observing that the above holds for any $(m-1)$-tuple $(i_{1}, \ldots, i_{m-1})$ and also for any $m$ satisfying $2 \leq m \leq k$.
 \end{proof}

\begin{lemma}\label{STPALemma3.2}
Let $A_1$, $A_2$, \ldots, $A_n$ be $n$ events. Given any $k\geq 2$, we have
$$\mathbb{P}(\text{at least}\hspace{2mm} k \hspace{1.5mm}\text{out of}\hspace{1.5mm} n \hspace{1.5mm} A_{i}\text{'s occur}) \leq \min_{1\leq m \leq k} \frac{S_m}{\binom{k}{m}}.$$
for each $1\leq m \leq k$.
\end{lemma}

\begin{proof}
    Let $T_n$ denotes the number of events occurring. Then, for each $1\leq m \leq k$,
\begin{align*}
     \mathbb{P}(\text{at least}\hspace{2mm} k \hspace{1.5mm}\text{out of}\hspace{1.5mm} n \hspace{1.5mm} A_{i}\text{'s occur}) = & \mathbb{P}(T_{n} \geq k) \\
     = & \mathbb{P}\bigg[\binom{T_n}{m} \geq \binom{k}{m}\bigg] \\
    \leq & \frac{1}{\binom{k}{m}} \mathbb{E}\bigg[ \binom{T_n}{m}\bigg] \\
    = & \frac{1}{\binom{k}{m}} \cdot \mathbb{E}\bigg[ \sum_{1 \leq i_1 < \cdots< i_m \leq n} I_{i_1}(w)I_{i_{2}}(w)\cdots I_{i_m}(w)\bigg] \\
    = & \frac{1}{\binom{k}{m}} \cdot \sum_{1 \leq i_1 < \cdots< i_m \leq n} \mathbb{P}(A_{i_{1}} \cap \cdots \cap  A_{i_m})\\
    = & \frac{S_m}{\binom{k}{m}}.
\end{align*}
The rest is obvious. 
\end{proof}
\noindent Thus we have the improved inequality:
\begin{theorem}\label{bound3}
    Let $A_1$, $A_2$, \ldots, $A_n$ be $n$ events. Suppose $S_{m}$ and $S^{\prime}_{m}$ are as mentioned earlier.  Then, for each $k \geq 2$,
$$\mathbb{P}(\text{at least}\hspace{2mm} k \hspace{1.5mm}\text{out of}\hspace{1.5mm} n \hspace{1.5mm} A_{i}\text{'s occur}) \leq \min \{A, B\}$$where
$$A:= \min_{2 \leq m \leq k} \frac{S_1-S^{\prime}_m}{k} + \frac{k-m+1}{k}\cdot \max_{i_1 < \ldots < i_{m-1}}\mathbb{P}(A_{i_1}\cap \ldots \cap A_{i_{m-1}}),$$
$$B:= \min_{1\leq m \leq k} \frac{S_m}{\binom{k}{m}}.$$
\end{theorem}
\autoref{bound3} is an extremely general probability inequality that works under any kind of joint behavior of events $A_1, \ldots, A_n$.

\section{Improved Bounds under Arbitrary Dependence\label{sec:5} and Equicorrelation}

Suppose $(X_{1}, X_{2}, \ldots, X_{n})$ have covariance matrix $\Sigma_n=((\rho_{ij}))$. We define $A_{i}=\{X_{i}>\Phi ^{-1}(1-k\alpha/n)\}$ for $1 \leq i \leq n$. This implies
$$k\text{-FWER}(n, \alpha,\Sigma_n)=\mathbb{P}_{\Sigma_{n}}\left(\text{at least} \hspace{1.5mm}k \hspace{1.5mm} A_{i}\text{'s occur} \mid H_{0}\right).$$
Hence, \autoref{bound3} gives the following immediate corollary:
\begin{corollary}\label{corollary1}
Let $\Sigma_n$ be the correlation matrix of $X_1, \ldots, X_n$ with $(i,j)$’th entry $\rho_{ij}$. Suppose $A_{i}=\{X_{i}>\Phi ^{-1}(1-k\alpha/n)\}$ for $1 \leq i \leq n$. Then, for each $k>1$, 
$$k\text{-FWER}(n, \alpha,\Sigma_n) \leq \min \{A,B\}$$
where
$A$ and $B$ are as mentioned earlier.  
\end{corollary}
\noindent We focus the equicorrelated case now. When $\rho_{ij} =\rho$ for each $i\neq j$, one has 
$$S_{m} = \binom{n}{m} \cdot a_m$$
where $a_m = \displaystyle\mathbb{P}\bigg[ \bigcap_{j=1}^{m} \{X_{i_j} > \Phi^{-1}(1-k\alpha/n) \}\bigg]$. This means 
$$\frac{S_m}{\binom{k}{m}} = \frac{\binom{n}{m} \cdot a_m}{\binom{k}{m}} = \frac{n!}{k!} \cdot \frac{(k-m)!}{(n-m)!} \cdot a_m.$$
Thus, 
\begin{align*}
    r_m:= \dfrac{\frac{S_{m+1}}{\binom{k}{m+1}}}{\frac{S_m}{\binom{k}{m}}}
    = \dfrac{n-m}{k-m}\cdot  \frac{a_{m+1}}{a_m}.
\end{align*}
We have the following:
\begin{theorem}
    Consider the equicorrelated normal set-up with correlation $\rho > 0$. Then, the sequence $r_m$ is increasing in $1\leq m\leq k-1$.
\end{theorem}
\begin{proof}
    Consider the block equicorrelation structure as mentioned in Appendix. Suppose $\mathbf{k}=(m+1,m-1,0,\ldots,0)$ and $\mathbf{k}^{\star}=(m,m,0,\ldots,0)$, where $2m-2$ zeros are there in each of these two vectors. So $\mathbf{k}>\mathbf{k}^\star$. Applying Theorem \ref{Tong}, we obtain
    $$a_{m+1}a_{m-1}\geq a_{m}^2.$$
This means $a_{m+1}/a_m$ is increasing in $m$. Since, $(n-m)/(k-m)$ is also increasing in $m$, we have the desired result. \end{proof}

\begin{lemma}\label{equi1}
Consider the equicorrelated normal set-up with correlation $\rho > 0$. Suppose $k>1$ and $k/n, \alpha<.5$. Then, $r_1<1$.
\end{lemma}
\noindent To prove this, we need the following representation theorem by \cite{Monhor_2013}:
\begin{lemma}\label{STPAlemma3.4}
 Suppose $(X,Y)$ follows a bivariate normal distribution with parameters $(0,0,1,1,\rho)$ with $\rho \geq 0$. Then, for all $x >0$, 
$$\mathbb{P}(X \leq x, Y \leq x)=[\Phi(x)]^2+\dfrac{1}{2\pi}\int_{0}^{\rho}\dfrac{1}{\sqrt{1-z^2}}e^{\frac{-x^2}{1+z}}dz.$$
\end{lemma}

\begin{proof}
We start with finding an expression on $a_2$:
\begin{align*}
    a_2 & = \mathbb{P}_{H_{0}}(A_{i} \cap A_{j})\\
    &=1 - \mathbb{P}_{H_{0}}(A_{i}^{c} \cup A_{j}^{c})\\
    &=1 - \mathbb{P}_{H_{0}}(A_{i}^{c}) - \mathbb{P}_{H_0}(A_{j}^{c}) + \mathbb{P}_{H_0}(A_{i}^{c} \cap A_{j}^{c})\\
    &=1-(1-k\alpha/n)-(1-k\alpha/n)+\mathbb{P}_{H_{0}}\bigg(X_i \leq \Phi ^{-1}(1-k\alpha/n), X_j \leq \Phi ^{-1}(1-k\alpha/n)\bigg)\\
    &=\frac{2k\alpha}{n}-1+(1-k\alpha/n)^2 +\dfrac{1}{2\pi}\int_{0}^{\rho_{ij}}\dfrac{1}{\sqrt{1-z^2}}e^{\frac{-\{\Phi ^{-1}(1-\frac{k\alpha}{n})\}\color{black}{}^2}{1+z}}dz \quad \text{(using Lemma \ref{STPAlemma3.4})}\\
    &=\frac{k^{2} \alpha^2}{n^2}+\dfrac{1}{2\pi}\int_{0}^{\rho}\dfrac{1}{\sqrt{1-z^2}}e^{\frac{-\{\Phi ^{-1}(1-\frac{k\alpha}{n})\}^2}{1+z}}dz.
\end{align*}
Hence, 
    \begin{align*}
        r_1<1 & \iff a_2 < \frac{k-1}{n-1}\cdot a_1\\
        & \iff \left(\frac{k\alpha}{n}\right)^2 + \dfrac{1}{2\pi}\int_{0}^{\rho}\dfrac{1}{\sqrt{1-z^2}}e^{\frac{-\{\Phi ^{-1}(1-\frac{k\alpha}{n})\}^2}{1+z}}dz < \frac{k-1}{n-1}\cdot \frac{k\alpha}{n} \\
        &\iff \dfrac{1}{2\pi}\int_{0}^{\rho}\dfrac{1}{\sqrt{1-z^2}}e^{\frac{-\{\Phi ^{-1}(1-\frac{k\alpha}{n})\}^2}{1+z}}dz < \frac{k\alpha}{n} \bigg[\frac{k-1}{n-1}- \frac{k\alpha}{n}\Bigg] \\
        & \impliedby \frac{\sin^{-1}(\rho)}{2\pi} \cdot e^{-\frac{c^2}{2}} < \frac{k\alpha}{n} \bigg[\frac{k-1}{n-1}- \frac{k\alpha}{n}\Bigg] \quad \text{(where $c=\Phi ^{-1}(1-k\alpha/n)$)} \\
        & \impliedby \frac{1}{2} \frac{k\alpha}{n} \bigg[\frac{k-1}{n-1}- \frac{k\alpha}{n}\Bigg] < c^2 \\
        & \impliedby 1/4<c \quad \text{(since $k/n, \alpha<.5$)} \\
        & \impliedby k\alpha/n < 1-\Phi(1/4)
    \end{align*}
which is true since since $k/n, \alpha<.5$.
\end{proof}
Thus, $r_{m}$ is increasing in $m$ and $r_1<1$. Hence, if there exists $m^{\star}$ such that $r_{m^{\star}-1}<1<r_{m^{\star}}$ then 
$$\min_{1\leq m \leq k} \frac{S_m}{\binom{k}{m}} = \frac{S_{m^{\star}}}{\binom{k}{m^{\star}}}.$$
\section{Results under Nearly Independence\label{sec:6}}
The correlation between two single nucleotide polymorphisms (SNP) is thought \citep{Proschan} to decrease with genomic distance. Many authors have argued that for the large numbers of markers typically used for a GWA study, the test statistics are weakly correlated because of this largely local presence of correlation between SNPs. \cite{StoreyTib2003} define weak dependence as “any form of dependence whose effect becomes negligible as the number of features increases to infinity” and remark that weak dependence generally holds in genome-wide scans. \cite{DasBhandari2025} consider the \textit{nearly independent setup}: 

$$\forall i \neq j, \quad
\operatorname{Corr}\left(X_i, X_j\right) =O\left(\frac{1}{n^\beta}\right) = \rho_{i j} \quad \text{(say)} \quad  \text { for some } \beta>0.
$$

\noindent Under the global null, we have
$$
\begin{aligned}
 k\text{-FWER}_{modified}\left(n, \alpha, \Sigma_{n}\right) & \leqslant \binom{n}{k} \cdot \max_{(i_1, \ldots, i_k)}\mathbb{P}\left[\bigcap_{j=1}^{k}\left\{X_{i_j}>c_{k, \alpha^{\star}, n}\right\}\right] 
\end{aligned}
$$
where $c_{k, \alpha^{\star}, n}=\Phi^{-1}(1-k\alpha^{\star}/n)$. 

Proceeding along the similar lines of \cite{DasBhandari2025} (see Lemma 4.3 therein), one obtains the following: 

$$\mathbb{P}\left[\displaystyle\bigcap_{j=1}^{k}\left\{X_{i_j}>c_{k, \alpha^{\star}, n}\right\}\right] \sim \displaystyle \frac{f(c_{k, \alpha^{\star}, n}, {\Sigma_{n}}^{-1})}{\prod_{i=1}^k \Delta_i} \quad \text{as $n \to \infty$}$$ provided that $c \to \infty$ as $n \to \infty$. Here $\Delta_{i} = \frac{c_{k, \alpha^{\star}, n}}{1+(k-1)\rho}$ for each $i$. 

We also have the following asymptotic approximation
$$\displaystyle \frac{f(c_{k, \alpha^{\star}, n}, {\Sigma_{n}}^{-1})}{\prod_{i=1}^k \Delta_i} \sim \left(\frac{k\alpha^{\star}}{n}\right)^k\left(1+\frac{c^2}{2} \displaystyle\sum_{l \neq m \in\left\{i_1, \ldots, i_k\right\}} \rho_{l m}\right) \quad \text{as $n \to \infty$.}$$

Combining these two approximations gives that under the global null, for all sufficiently large values of $n$,
$$
\begin{aligned}
 k\text{-FWER}_{modified}\left(n, \alpha, \Sigma_{n}\right) & \leqslant \binom{n}{k} \left(\frac{k\alpha^{\star}}{n}\right)^k \cdot \max_{(i_1, \ldots, i_k)}\left(1+\frac{c^2}{2} \displaystyle\sum_{l \neq m \in\left\{i_1, \ldots, i_k\right\}} \rho_{l m}\right).
\end{aligned}
$$
\section{Empirical Study\label{sec:7} and Real Data Analysis}
The set-up of our empirical study is as follows: 
\begin{itemize}
    \item number of hypotheses: $n=1000$. 
    \item choice of $k$: $k=25,50,75$. 
    \item choice of equicorrelation: $\rho \in \{.1, .15, .2, .25, .3\}$. 
    \item desired level of control : $\alpha=.05$. 
       \end{itemize}

       We present the empirical results for $k$-FWER for $n =
 1000, \alpha = .05$ and the aforementioned choices of $(n, k, \rho, \alpha)$ along with the existing bounds by \cite{deybhandaristpa} and our proposed bounds (given by Corollary \ref{corollary1}) in Table \ref{tab1}. It is noteworthy that for very small values of $\rho$, our proposed bounds are
 significantly smaller than the existing bounds, while the proposed bounds are never worse than the existing ones.
 
\begin{table}[ht]
\begin{center}
\caption{Estimates of $k$-FWER($n=1000,\alpha=.05,\rho$)}
\label{tab1}
\begin{tabular}{  c c c c c c c } 
 \hline
 $k$  & $\rho$ & 0.1 & 0.15 & 0.2 & 0.25 & 0.3 \\ \hline 
 25  & $\hat{k\text{-FWER}(\rho)}$ & 1e-04 & 0.0013 & 0.0024 & 0.0050 & 0.0062 \\ 
 & Existing  Bound  & 0.007098 & 0.011068 & 0.01672 & 0.02456 & 0.03522 \\ 
  & Proposed Bound  & 0.001392 & 0.006829 & 0.01672 & 0.02456 & 0.03522 \\ \hline
 50  & $\hat{k\text{-FWER}(\rho)}$  & 0e+00 & 0.0006 & 0.0017 & 0.0032 & 0.0048 \\
   & Existing Bound & 0.006185 & 0.009164& 0.01321&0.01857 &0.02557  \\
   & Proposed Bound  & 0.000447 &0.003484 & 0.00962 & 0.01857 &0.02557 \\ \hline
 75  & $\hat{k\text{-FWER}(\rho)}$  & 0e+00 & 0.0003 & 0.0013 & 0.0025 & 0.0037  \\
   & Existing Bound & 0.005739 & 0.008254 & 0.01157 &0.01587 &0.02134  \\ 
   & Proposed Bound  & 0.000191 & 0.002094 & 0.00681 & 0.01374 & 0.02134 \\ \hline
\end{tabular}
\end{center}
\end{table}

We now elucidate the practical utility of our proposed bounds through a \textit{prostate cancer dataset} \citep{Singh2002}. This dataset involves expression levels of $n=6033$ genes for $N=102$ individuals: among them 52 are prostate cancer patients and the rest are healthy persons.  

Let $T_{ij}$ denote the expression level for gene $i$ on individual $j$. The dataset is a $n \times N$ matrix $\mathbf{T}$ with $1\leq j \leq50$ for the healthy persons and $51\leq j \leq102$ for the cancer patients. Let $\bar{T}_{i}(1)$ and $\bar{T}_{i}(2)$ be the averages of $T_{i j}$ for these two groups respectively. one of the foremost problem is to to identify genes that have significantly different levels among the two populations \citep{Efron2010book}. In other words, one is interested in testing
\begin{align*}
H_{0 i} \text { : } & T_{i j} \text{\hspace{1.5mm}has the same distribution for the two populations}. 
\end{align*}

A natural way to test $H_{0i}$ is to compute the usual $t$ statistic $t_{i}=\frac{\bar{T}_{i}(2)-\bar{T}_{i}(1)}{S_{i}}$. Here, 
$$
S_{i}^{2}=\displaystyle \frac{\displaystyle \sum_{j=1}^{50}\left(T_{i j}-\bar{T}_{i}(1)\right)^{2}+\sum_{j=51}^{102}\left(T_{i j}-\bar{T}_{i}(2)\right)^{2}}{100} \cdot\left(\frac{1}{50}+\frac{1}{52}\right).
$$

One rejects $H_{0i}$ at $\alpha=.05$ (based on usual normality assumptions) if $\left|t_{i}\right|$ exceeds 1.98, i.e, the two-tailed 5\% point for a Student-$t$ random variable with 100 d.f. We consider the following transformation: 
$$X_{i}=\Phi^{-1}\left(F\left(t_{i}\right)\right),$$
\noindent where $F$ denotes the cdf of $t_{100}$ distribution. This means
$$H_{0 i}: X_{i} \sim N(0,1).$$
The observed value of the usual correlation coefficient efficient between the $n$ $X$ values is less than $.001$. Hence, we take $\rho=0$ in our computations.

For a given $k$, the Lehmann-Romano procedure at level $\alpha=.05$ rejects $H_{0i}$ if $X_i > \Phi^{-1}(1-k\alpha/n)$. Our proposed procedure rejects $H_{0i}$ if $X_i> \Phi^{-1}(1-k\alpha^{\star}/n)$ where $\alpha^{\star}=\left[\frac{n}{k \cdot\binom{n}{k}^{1 / k}}\right] \alpha^{1 / k}$.  The numbers of rejected hypotheses for different values of $k$ by Lehmann-Romano procedure, the method proposed in \cite{deybhandaristpa}, and our proposed procedure are given in Table \ref{tab0}.

\begin{table}[ht]
\begin{center}
\caption{Number of rejected hypotheses for different values of $k$}
\label{tab0}
\begin{tabular}{  c c c c c c c c c c c} 
 \hline
 $k$  & $2$ & 10 &20 &40 &60 \\ \hline 
 Lehmann-Romano Method  & 6 & 13 &17 &26 &27 \\ \hline
 Method in \cite{deybhandaristpa} & 12  & 26 &34 &51 &62 \\  \hline
 Proposed Method  & 12  & 28  &47 &63 &78 \\  \hline
\end{tabular}
\end{center}
\end{table}
Table \ref{tab0} demonstrates that our proposed method rejects significantly more number of hypotheses than the method proposed in \cite{deybhandaristpa} and the Lehmann-Romano method for each value of $k$. This illustrates the superiority of our proposed bound on the probability of occurrence of at least $k$ among $n$ events.

\section{Concluding Remarks\label{sec:8}}

Computing generalized familywise error rates involves the knowledge of joint distribution of test statistics under null hypotheses. While this distribution is relatively straightforward under independence, this becomes intractable under dependence, especially arbitrary or unknown dependence. While dependence is a natural phenomenon in a plethora of scientific avenues, the nature and extent of dependence varies with the context and complexity of the particular scenario. 
This paper revisits the classical testing problem of normal means in different correlated frameworks. We establish upper bounds on the generalized familywise error rates under each dependence, consequently giving rise to improved testing procedures. Towards this, we also present improved inequalities on the
 probability that at least $k$ among $n$ events occur, which are of independent theoretical interest. This probability arises, e.g., in transportation and communication networks. The new probability inequalities proposed in this work might be insightful
 in those areas, too.

\section*{Acknowledgements}
Dey is grateful to Prof. Thorsten Dickhaus for his constant encouragement and support during this work.

\appendix

\bibliography{references}

\section{Appendix}
Let $k_1, \ldots, k_n$ be nonnegative integers such that $\sum_{i=1}^n k_i=n$, and denote $\mathbf{k}=\left(k_1, \ldots, k_n\right)^{\prime}$. Without loss of generality, it may be assumed that
$$
k_1 \geq \cdots \geq k_r>0, \quad k_{r+1}=\cdots=k_n=0
$$
for some $r \leq n$. \cite{Tong} defines $r$ square matrices $\boldsymbol{\Sigma}_{11}, \ldots, \boldsymbol{\Sigma}_{rr}$ such that $\mathbf{\Sigma}_{jj}(k)$ has 1 on diagonals and $\rho$ on each each off-diagonals for $j=1, \ldots, r$. We define an $n \times n$ matrix $\mathbf{\Sigma}(\mathbf{k})$ given by
$$
\mathbf{\Sigma}(\mathbf{k})=\left(\begin{array}{cccc}
\boldsymbol{\Sigma}_{11} & \mathbf{0} & \cdots & \mathbf{0} \\
\mathbf{0} & \boldsymbol{\Sigma}_{22} & \cdots & \mathbf{0} \\
& \cdots & & \\
\mathbf{0} & \mathbf{0} & \cdots & \boldsymbol{\Sigma}_{r r}
\end{array}\right).
$$

Let $\mathbf{X}=\left(X_1, \ldots, X_n\right)^{\prime}$ have an $N_n(\mu, \mathbf{\Sigma}(\mathbf{k}))$ distribution. For fixed block sizes $k_1, \ldots, k_r, X_i$ and $X_j$ are correlated with a common correlation coefficient $\rho$ if they are in the same block, and are independent if they belong to different blocks. The problem of interest is how the distributions and moments of the extreme order statistics $X_{(1)}$ and $X_{(n)}$ depend on the block size vector $\mathbf{k}$.

Let $\mathbf{k}^\star=\left(k_1^\star, \ldots, k_n^\star\right)^{\prime}$ denote another vector of nonnegative integers such that
$$
\quad k_1^\star \geq \cdots \geq k_r^\star>0, \quad k_{r+1}^\star=\cdots=k_n^\star=0
$$
and $\sum_{i=1}^{r^\star} k_i^\star=n$. Let $\mathbf{\Sigma}\left(\mathbf{k}^\star\right)$ denote the similar matrix as $\mathbf{\Sigma}(\mathbf{k})$ but based on $\mathbf{k}^\star$ instead of $\mathbf{k}$. 

One has the following:

\begin{theorem}(see Theorem 6.4.5 of \cite{Tong})\label{Tong}
    Let $\left(X_1, \ldots, X_n\right)^{\prime}$ have an $N_n(\mu, \mathbf{\Sigma}\left(\mathbf{k}\right))$ distribution, let $\left(X_1^\star, \ldots, X_n^\star\right)^{\prime}$ have an $N_n\left(\mu, \mathbf{\Sigma}\left(\mathbf{k}^\star\right)\right)$ distribution, and let $$
X_{(1)} \leq \cdots \leq X_{(n)}, \quad X_{(1)}^\star \leq \cdots \leq X_{(n)}^\star
$$
denote the corresponding order statistics. If $\mu_1=\cdots=\mu_n$ and $\mathbf{k}>\mathbf{k}^\star$, then:

$$\quad P\left[X_{(1)} \geq \lambda\right] \geq P\left[X_{(1)}^\star \geq \lambda\right] \quad \forall\lambda.$$
\end{theorem}

\end{document}